\documentclass[reqno, 12pt]{amsart}
\usepackage[letterpaper,hmargin=1in,vmargin=1in]{geometry}
\usepackage{amsmath, amssymb, amsthm, verbatim, bbm, url} 
\usepackage{graphicx}
\usepackage{enumerate}

\newcommand \R {\mathbb{R}}
\newcommand \C {\mathbb{C}}

\newcommand \HH {\mathbb{H}}
\newcommand \BB {\mathbb{B}}
\newcommand \RR {\mathbb{R}}

\newcommand \Oh {\mathcal{O}}

\newcommand \la {\langle}
\newcommand \ra {\rangle}

\newcommand \D {\partial}
\newcommand \eps {\varepsilon}

\newcommand \Def {\stackrel{\textrm{def}}=}
\newcommand\comp{{\mathrm{comp}}}
\newcommand\Diag{{\mathrm{Diag}}}
\newcommand\CI {C^\infty}
\newcommand\even{{\mathrm{even}}}
\newcommand\pa{\partial}
\newcommand\cR{\mathcal{R}}
\newcommand\Vf{\mathcal{V}}
\newcommand\olX{\overline X}
\DeclareMathOperator \re {Re}
\DeclareMathOperator \im {Im}

\DeclareMathOperator \supp {supp}

\DeclareMathOperator \WFh {WF_{\textit{h}}}
\DeclareMathOperator \Op {Op}
\DeclareMathOperator \Id {Id}

\DeclareMathOperator \Diff {Diff}

\newtheorem{lem}{Lemma}
\newtheorem{thm}{Theorem}
\newtheorem{cor}{Corollary}

\theoremstyle{definition}

\newtheorem{rem}{Remark}
\newtheorem{Defn}{Definition}
\numberwithin{equation}{section}
\numberwithin{lem}{section}
\numberwithin{cor}{section}
\numberwithin{rem}{section}
\numberwithin{Defn}{section}
\numberwithin{thm}{section}
\title
[Gluing semiclassical resolvent estimates]
{Gluing semiclassical resolvent estimates via propagation of singularities}

\author[Kiril Datchev]
{Kiril Datchev}
\address{Department of Mathematics, Massachusetts Institute of Technology, Cambridge, MA
02139-4397, U.S.A.}
\email{datchev@math.mit.edu}
\author[Andr\'as Vasy]
{Andr\'as Vasy}
\address{Department of Mathematics, Stanford University, Stanford, CA
94305-2125, U.S.A.}
\email{andras@math.stanford.edu}
\thanks{The first author is partially supported by a National Science Foundation postdoctoral fellowship, and the second author is
partially supported by
the National Science Foundation under
grant DMS-0801226, and
a Chambers Fellowship from Stanford University.}
\date{November 19, 2011}
\begin{document}

\begin{abstract}
We use semiclassical propagation of singularities to give a
general method for gluing together resolvent estimates. As an
application we prove estimates for the analytic continuation of the
resolvent of a Schr\"odinger operator for certain asymptotically
hyperbolic manifolds in the presence of trapping which is sufficiently
mild in one of several senses. As a corollary we obtain local
exponential decay for the wave propagator and local smoothing for the
Schr\"odinger propagator.
\end{abstract}
\maketitle

\section{Introduction}

In this paper we give a general method for gluing semiclassical resolvent estimates. As an application we obtain the following theorem.
\begin{thm}\label{t:intro}
Let $(X,g)$ be an even asymptotically hyperbolic Riemannian manifold. Let
\[P = h^2 \Delta_g -1, \qquad 0 < h \le h_0.\]
 Suppose that the trapped set of $X$, i.e. the set of maximally extended geodesics which are precompact, is either normally hyperbolic in the sense of \S \ref{s:wz} or hyperbolic with negative topological pressure at $1/2$ (see \S \ref{s:nz}). Then the cutoff resolvent $\chi(P- \lambda)^{-1} \chi$ continues analytically from $\{\im \lambda > 0\}$ to $[-E,E]-i[0,\Gamma h]$ for every $E \in (0,1)$ and $\chi \in C_0^\infty(X)$, where it obeys the bound
 \[\|\chi(P- \lambda)^{-1} \chi\|_{L^2 \to L^2} \le a(h).\]
Here $h^{-1} \lesssim a(h) \lesssim h^{-N}$ and $\Gamma>0$ are both determined only by the trapped set. Moreover, for $E' \in [-E,E]$ we have for some $C>0$ the following quantitative limiting absorption principle:
 \[\|\chi(P- E' - i0)^{-1} \chi\|_{L^2 \to L^2} \le C \log(1/h)h^{-1}.\]
\end{thm}

By an \textit{even asymptotically hyperbolic manifold} we mean that $X$ is the interior of $\overline X$, a compact manifold with boundary, and
\[
g = \frac{dx^2 + \tilde g}{x^2}, \textrm{ near }\D \overline{X}.
\]
Here $x \in C^\infty(\overline{X})$ is a boundary defining function and $\tilde g$ is a family of metrics on $\D\overline{X}$, smooth up to $\D\overline{X}$,\footnote{We can reduce a more general case to this one. Namely, it suffices to assume that $\tilde g$ is a 2-cotensor which is a metric on $\D\overline{X}$ when restricted to $\D \overline{X}$: see \cite[Proposition 2.1]{js}.} and even in $x$ (see \cite[Definition 1.2]{g} for a more invariant way to phrase this last condition). The assumptions on the trapped set will be discussed in more detail in \S \ref{s:trapped}, but for now we remark that it suffices to take $X$ negatively curved with a trapped set consisting of a single closed geodesic.  Moreover, the compact part of the manifold on which the trapping occurs can be replaced by a domain with several convex obstacles. See \S\ref{s:apps} for a stronger result.

As already stated our methods work in much greater generality. Let $(X,g)$ be a complete Riemannian manifold, $P=h^2\Delta_g+V-1$
a semiclassical Schr\"odinger operator, $V\in\CI(X;\RR)$ bounded, $h\in(0,1)$. Then
$P$ is essentially self-adjoint, 
$R(\lambda)=(P-\lambda)^{-1}$ is holomorphic in $\{\lambda:\ \im\lambda\neq 0\}$.
Moreover, in this set
one has uniform estimates on $R(\lambda):L^2\to L^2$ as $h\to 0$, namely $\|R(\lambda)\| \le 1/|\im \lambda|$.

On the other hand,
as $\lambda$ approaches the spectrum, $R(\lambda)$ is necessarily not uniformly
bounded (even for a single $h$). However, in many settings, e.g. on asymptotically
Euclidean or hyperbolic spaces with a suitable decay assumption on $V$, the resolvent extends continuously to the
spectrum (perhaps away from some thresholds), although only as an operator on weighted
$L^2$-spaces. Indeed under more restrictive assumptions it continues
meromorphically across the continuous spectrum, typically to a Riemann surface
ramified at thresholds: see e.g. \cite{mel} for a general discussion.

It is very useful in this setting to obtain semiclassical
resolvent estimates, i.e.\ estimates as $h\to 0$, both at the spectrum of $P$ and for
the analytic continuation of the resolvent, $R(\lambda)$. By scaling, these imply high energy resolvent
estimates for non-semiclassical Schr\"odinger operators, which in turn can
be used, for instance, to describe wave propagation, or more precisely the
decay of solutions of the wave equation: see \S\ref{s:apps} for examples of such applications in our setting. For this purpose the most relevant
estimates are those in a strip near the real axis for the non-semiclassical problem
(which gives exponential decay rates for the wave equation),
which translates to estimates in an $\Oh(h)$ neighborhood of the real axis for
semiclassical problems.

The best estimates one can expect (on appropriate
weighted spaces) are $\Oh(h^{-1})$; this corresponds to the fact that this problem is
semiclassically of real principal type. However, if there is trapping, i.e.\ some trajectories of
the Hamilton flow (or, in case $V=0$, geodesics) do not escape to infinity, the
estimates can be significantly worse (exponentially large: see \cite{Burq:Lower}) even at the real axis.
Nonetheless, if the trapping is mild, e.g.\ the trapped set is hyperbolic, then
one has polynomial, $\Oh(h^{-N})$, bounds in certain settings, see the work of
Nonnenmacher-Zworski \cite{Nonnenmacher-Zworski:Quantum, nz2}, Petkov-Stoyanov \cite{ps}, and Wunsch-Zworski \cite{wz}. Moreover, on the real axis (i.e. for $R(\lambda + i0)$ with $\lambda \in \R$) \cite{Nonnenmacher-Zworski:Quantum, wz} prove $\Oh(\log(1/h)h^{-1})$ bounds, which work of Bony-Burq-Ramond \cite{Bony-Burq-Ramond} shows to be optimal when any trapping is present.

That $\Oh(h^{-1})$ bounds hold in strips for nontrapping asymptotically hyperbolic manifolds was proved by the second author in \cite{v1} (see also \cite{v2} and \cite{Melrose-SaBarreto-Vasy:Semiclassical}). This result, along with the accompanying propagation of singularities theorem, is what makes it possible to use our gluing method to prove Theorem \ref{t:intro}; see \S\ref{s:ahinfinity} for a discussion of how to. 

Typically, for the settings in which one can prove polynomial bounds in the presence
of trapping, one considers particularly convenient models in which one alters
the problem away from the trapped set, e.g.\ by adding a complex absorbing
potential. The natural expectation is that if one can prove such bounds in a thus
altered setting, one should also have the bounds if one alters the operator in
a different non-trapping manner, e.g.\ by gluing in a Euclidean end or another
non-trapping infinity. In spite of this widespread belief, no general
prior results exist in this direction, though in some special cases this has been
proved using partially microlocal techniques, e.g.\ in work of Christianson \cite{c07,c08} on resolvent estimates where the trapping consists of a single hyperbolic orbit, and in the work of the first author
\cite{d}, combining the estimates of Nonnenmacher-Zworski \cite{Nonnenmacher-Zworski:Quantum, nz2} with the microlocal non-trapping
asymptotically Euclidean estimates of the second author and
Zworski \cite{vz} as well as the more delicate non-microlocal estimates
of Burq \cite{Burq:Lower} and Cardoso-Vodev \cite{Cardoso-Vodev:Uniform}. Another example is work of the first author \cite{d2} using an adaptation of the method of complex scaling to glue in another class of asymptotically hyperbolic ends to the estimates of Sj\"ostrand-Zworski \cite{sz} and \cite{Nonnenmacher-Zworski:Quantum,nz2}.
It is important to point out, however, that in the present paper we glue the resolvent estimates {\em directly}, without the need for any information on how they were obtained, so for instance we do not need to construct a global escape function, etc. In addition to the above listed
references, Bruneau-Petkov \cite{bp} give a general method for deducing weighted resolvent estimates from cutoff resolvent estimates, but they require that the operators in the cutoff estimate and the weighted estimate be the same.

In this paper we show how one can achieve this gluing in general, in a robust
manner. The key point is the following.
One cannot simply use a partition of unity to combine the trapping model with
a new `infinity' because the problem is not semiclassically elliptic.
Thus, semiclassical singularities (i.e. lack of decay as $h\to 0$) propagate
along null bicharacteristics. However, under a convexity assumption on the gluing
region, which holds for instance when gluing in asymptotically Euclidean or
hyperbolic models, following these singularities microlocally allows us to show
that an appropriate three-fold iteration of this construction, which
takes into account the bicharacteristic flow, gives a parametrix with $\Oh(h^\infty)$
errors. This in turn allows us to show that the resolvent of the glued operator satisfies
a polynomial estimate, and that on asymptotically Euclidean and hyperbolic manifolds the order of growth is given by that of the model operator for the trapped region.

We state our general assumptions and main result precisely in the next section, and prove the result in \S \ref{s:proof}. In \S \ref{s:infinity} we will show how our assumptions near infinity are satisfied for various asymptotically Euclidean and asymptotically hyperbolic manifolds. In \S\ref{s:trapped} we show how our assumptions near the trapped set are satisfied for various types of hyperbolic trapping. In \S \ref{s:apps} we give applications: a more precise version of Theorem \ref{t:intro}, exponential decay for the wave equation, and local smoothing for the Schr\"odinger propagator.

We are very grateful to Maciej Zworski for his suggestion, which started this project, that resolvent gluing should be understood much better, and for his interest in this project.
Thanks also to both anonymous referees for their useful comments.

\section{Main theorem}\label{s:abstract}

Before stating the abstract assumptions of our main theorem, we review some terminology from semiclassical analysis.
For $a \in C^\infty(T^*X)$ a symbol supported in a coordinate patch, $\psi\in\CI(X)$
compactly supported in the patch, $\psi\equiv 1$
on a neighborhood of the projection to $X$ of the support of $a$,
$\Op(a)$ is a semiclassical quantization given in local coordinates by
\[\Op(a)u(z) = \frac 1 {(2\pi h)^n} \int e^{iz\zeta/h}a(z,\zeta)\widehat {(\psi u)}(\zeta)d\zeta.\]
See, for example, \cite{ds,ez} for more information. We also say
that a family of functions $u=(u_h)_{h\in (0,1)}$ on $X$
is polynomially bounded if $\|u\|_{L^2}\lesssim h^{-N}$ for some $N$.
The semiclassical wave front set, $\WFh(u)$, is defined for polynomially bounded
$u$ as follows: for $q\in T^*X$, $q\notin\WFh(u)$ iff there exists
$a\in\CI_0(T^*X)$ with $a(q) \ne 0$ such that $\Op(a)u=\Oh(h^\infty)$ (in $L^2$). One can also
extend the definition to $q\in S^*X$ (thought of as the cosphere bundle at
fiber-infinity in $T^*X$) by considering $a \in C_0^\infty(\overline{T}^*X)$, where $\overline{T}^*X$ is the fiber-radial compactification of $T^*X$ with $a(q) \ne 0$ for $q \in S^*X$;
then $\WFh(u)=\emptyset$ implies $u=\Oh(h^\infty)$ (in $L^2$).

Let $\overline X$ be a compact manifold with boundary, $g$ a complete metric on $X$ (the interior of $\overline X$), and $P = h^2 \Delta_g + V$ a semiclassical Schr\"odinger operator on $X$ with $V \in C^\infty(X;\R)$. We have no further explicit conditions on the potential $V$, although the dynamical assumption \eqref{eq:convexity} and the assumptions on the resolvents below will in practice restrict the potentials the theorem applies to. Let $x$ be a boundary defining function, and let
\[X_0 \Def \{0 < x < 4\}, \qquad X_1 \Def \{x > 1\}.\]
Recall that a bicharacteristic of $P$ is an integral curve in $T^*X$ of the Hamiltonian vector field associated to the Hamiltonian function $p = |\xi|_g^2 + V(x)$, and that the energy of a bicharacteristic is the level set of $p$ in which it lies.
Suppose that the bicharacteristics $\gamma$ of $P$ (by this we always mean bicharacteristics at energy in some fixed range $[-E,E]$) satisfy the convexity
assumption
\begin{equation}\label{eq:convexity}
\dot x(\gamma(t))=0\Rightarrow \ddot x(\gamma(t))<0,
\end{equation}
in $X_0$. In particular this rules out the possibility of any trapped trajectory entering $X_0$. There may be more complicated behavior, including trapping, in $X \setminus X_0$.

\begin{rem}\label{rem:constants}
If $x$ is a boundary defining function, $f$ is a $\CI$
function on $[0,\infty)$ with $f'>0$,
and $x$ satisfies \eqref{eq:convexity} then so does
$f\circ x$.
In particular the specific constants above
(as well as intermediate constants used below) are chosen only for convenience,
and can be replaced by arbitrary constants so long as the ordering of the constants is preserved.
\end{rem}

Let $P_0$ and $P_1$ be model differential operators on $X_0$ and $X_1$ respectively: 
\begin{equation}\label{e:modelagree}
P|_{X_0} = P_0|_{X_0}, \textrm{ and } P|_{X_1} = P_1|_{X_1}.
\end{equation}
The spaces $X'_j$ on which the operators $P_j$ are globally defined can differ from $X$ away from $X_j$, as the operators will always be multiplied by a smooth cutoff function to the appropriate $X_j$. Assume however that no bicharacteristic of $P_1$ leaves $X_1$ and then returns later, i.e. that
\begin{equation}\label{eq:X_1-convex}
X_1\ \text{is bicharacteristically convex in}\  X_1'.
\end{equation}
Note that we do not assume that the $P_j$ are self-adjoint; this is useful in the applications in \S\S\ref{s:trapped}--\ref{s:apps}. However, the condition \eqref{e:modelagree} makes $P_j$ formally self-adjoint on $X_j$ and in particular has real semiclassical symbol on $T^*X_j$, as a result of which bicharacteristics of $P_j$ are well defined on $T^*X_j$. Note that one difference from the (in some ways related) ``black-box'' approach of Sj\"ostrand-Zworski \cite{szcomp} is that for us $X_1'$ will typically not be a compact manifold.

\begin{rem}
In some applications we may wish to divide $T^*X$, rather than simply $X$, into two overlapping regions with a different model operator on each one. To do this it is enough to take $x$ to be a more general function in $C^\infty(T^*X)$, rather than a boundary defining function for $\overline{X}$ as we do here. In that case $\Oh(h^\infty)$ error terms appear in the formulas in \eqref{e:modelagree}, and in several other formulas below, but the construction is essentially the same. For simplicity of exposition we do not pursue this level of generality further below.
\end{rem}

Let $\rho_0 \in C^\infty(X_0')$, $\rho_1 \in C^\infty(X_1')$ be bounded functions, referred to as \textit{weights} (in typical applications they may be compactly supported, decaying at infinity, or constant) such that $\rho_0 = 1$ on $X_0 \cap X_1$ and $\rho_1 = 1$ on $X_1$. Let $\rho \in C^\infty(X)$ have $\rho = \rho_0$ on $X_0$ and $\rho = 1$ otherwise.

Let $\rho R(\lambda)\rho $ denote the weighted resolvent $\rho (P-\lambda)^{-1} \rho $ in $\{\im \lambda > 0\}$ or its meromorphic continuation where this exists in $\{\im \lambda \le 0\}$, and similarly $\rho_0 R_0(\lambda) \rho_0 $ and $\rho_1 R_1(\lambda) \rho_1$ where those weighted resolvents exist. Suppose that $\tilde h_0\in (0,1)$ and
the $\rho_j R_j(\lambda)\rho_j$ continue from $\{\im \lambda > 0\}$ to  
\[\lambda \in D \subset [-E,E]-i[0,\Gamma h], \quad 0<h<\tilde h_0,\]
(note that $D$ need not be open, and may in particular consist of $[-E,E]+i0$),
and in that region they obey
\begin{equation}\label{eq:model-bound}
\|\rho_j R_j(\lambda)\rho_j\| \le a_j(h,\lambda) \lesssim h^{-N},\ 0<h<\tilde h_0,
\end{equation}
for some $a_j(h) \ge h^{-1}$. We call a resolvent satisfying \eqref{eq:model-bound} {\em polynomially bounded.}

\begin{Defn}\label{Def:outgoing}
Let $\lambda \in D$ and $q\in T^*X_j$ be in the characteristic set of $P_j - \lambda$, that is the zero set of $p_j - \re\lambda$, where $p_j$ is the semiclassical principal symbol of $P_j$. Let $\gamma_-:(-\infty,0]\to T^*X'_j$ (or $\gamma_-:(-t_q,0]\to T^*X'_j$ in case this is not defined for all time) be the
backward $P_j$-bicharacteristic from $q$.
We say that the resolvent $R_j(\lambda)$ is {\em semiclassically outgoing at $q$}
if
\begin{gather}\label{eq:outgoing-hyp}
u \in L^2_{\comp}(X_j)\ \text{polynomially bounded},\ \WFh(u)\cap\gamma_-=\emptyset
\end{gather}
implies that
\begin{gather}\label{eq:outgoing-concl}
q\notin\WFh(R_j(\lambda)u).
\end{gather}

We say that the resolvent $R_j(\lambda)$
is {\em off-diagonally semiclassically outgoing at $q$}
if \eqref{eq:outgoing-concl} holds provided we add
$q\notin T^*\supp u$ to the hypotheses \eqref{eq:outgoing-hyp}.
\end{Defn}

\begin{rem}\label{rem:outgoing}
Since $u$ is compactly supported,
the definition involves only the cutoff resolvent.

{\em In this paper we only need to make the assumption that
the resolvents $R_j(\lambda)$
are off-diagonally semiclassically outgoing. However, for brevity, we
use the term `semiclassically outgoing' in place of 
`off-diagonally semiclassically outgoing' throughout the paper.}

A reason for making the weaker hypothesis of being off-diagonally
semiclassically outgoing is that it is sometimes easier to check: typically the Schwartz kernel of $R_j(\lambda)$ is
simpler away from the diagonal than at the diagonal (where it may be a semiclassical
paired Lagrangian distribution), and the off-diagonal outgoing property follows
easily from the oscillatory nature of the Schwartz kernel there; see
the third paragraph of \S\ref{s:infinity}.
\end{rem}

Our microlocal assumption is then that
\begin{enumerate}
\item[(0-OG)]
$R_0(\lambda)$ is semiclassically outgoing at all $q\in T^*(X_0\cap X_1)$
(in the characteristic
set of $P_0$),
\item[(1-OG)]
$R_1(\lambda)$ is semiclassically outgoing at all $q\in T^*(X_0\cap X_1)$
(in the characteristic
set of $P_1$) such that $\gamma_-$ is disjoint from
$T^*(X'_1\setminus(X\setminus X_0)) = T^*(X'_1 \cap \{x >4\})$, thus disjoint from any trapping in $X_1$.
\end{enumerate}
In fact, for $R_0(\lambda)$, it is sufficient to have the property in
Definition~\ref{Def:outgoing} for $u\in L^2(X_0\cap X_1)$ (i.e.\ $u$ supported
in $X_0\cap X_1$).

The main result of the paper is the following general theorem.
\begin{thm}\label{t:main}
Under the assumptions of this section, there exists $h_0\in (0,1)$ such that
for $h<h_0$, $R(\lambda)$ continues analytically to $D$, where it obeys the bound
\[\|\rho R(\lambda) \rho\|\le Ch^2a_0^2a_1.\]
\end{thm}
In particular, when $a_0 = C/h$, we find that $\rho R(\lambda) \rho$ obeys (up to constant factor) the same bound as $\rho_1 R_1(\lambda) \rho_1$, the model operator with infinity suppressed.

\begin{rem}
The only way we use convexity is to argue that no bicharacteristics of $P$ go from $\{x>2+ \eps\}$ to $\{x<2\}$ and back to $\{x> 2+\eps\}$ for some $\eps > 0$. This is also fulfilled in some settings in which the stronger condition \eqref{eq:convexity} does not hold, for example when $X$ has cylindrical or asymptotically cylindrical ends. In particular, some mild concavity is allowed.
\end{rem}

\section{Proof of main theorem}\label{s:proof}

Let $\chi_1 \in C^\infty(\R;[0,1])$ be such that $\chi_1 = 1$ near $\{x \ge 3\}$, and $\supp \chi_1 \subset \{x > 2\}$, and let $\chi_0 = 1 - \chi_1$.

Define a right parametrix for $P$ by
\begin{equation}\label{e:fdef}
F \Def \chi_0(x-1)  R_0(\lambda)  \chi_0(x) + \chi_1(x+1)  R_1(\lambda) \chi_1(x)
\end{equation}
We then put
\[PF = \Id + [P,\chi_0(x-1)]R_0(\lambda) \chi_0(x) + [P,\chi_1(x+1)]R_1(\lambda)\chi_1(x) \Def \Id + A_0 + A_1.\]

This error is large in $h$ due to semiclassical propagation of
singularities: in general we have only $\|A_j\rho_j\|_{L^2 \to L^2}
\le C h a_j(h)$, and $h a_j(h) \ge 1$, and thus $\|A_j\rho_j\|$ typically
does not go to 0 with $h$. (Note that there is no need for a weight on
the left of $A_j$ in view of the support of the cutoffs.)
However, using an iteration argument we can replace it by a small error. Observe that by disjointness of supports of $d\chi_0(.-1)$ and $\chi_0$, resp.\ $d\chi_1(.+1)$ and $\chi_1$, we have
\begin{equation}\label{e:a2}A_0^2 = A_1^2 = 0,\end{equation}
while Lemma \ref{zigzag:worse} below implies that
\begin{equation}\label{e:a0a1}\|A_0A_1\|_{L^2 \to L^2} = \Oh(h^\infty).\end{equation}
This is the step in which we exploit the semiclassical propagation of singularities (see Figure \ref{f:noreturn}). Note that we have
\[
A_0A_1 =  [P,\chi_0(x-1)]\rho_0R_0(\lambda)  \rho_0 [P,\chi_1(x+1)]\rho_1 R_1(\lambda)\rho_1\chi_1(x),
\]
that is to say, inserting weights $\rho_0$ and $\rho_1$ amounts to multiplying by $1$ thanks to the cutoff functions $\chi_j$ which are present.

\begin{lem}\label{zigzag:worse}
Suppose that $\varphi_1,\varphi_2,\varphi_3$ are compactly supported semiclassical differential operators  (i.e. given in local coordinates by $\sum_\alpha a_\alpha(z) h^{|\alpha|}D_z^\alpha$ where the sum is over a finite set of multiindices $\alpha$) with
\[\supp \varphi_1 \subset \{2 < x\}, \quad \supp \varphi_2 \subset \{1 < x < 2\}, \quad \supp \varphi_3 \subset \{3 < x < 4\}.\]
Then
\[\|\varphi_3 R_0(\lambda) \varphi_2 R_1(\lambda) \varphi_1 \|_{L^2 \to L^2} = \Oh(h^\infty).\]
\end{lem}

\begin{figure}
\includegraphics{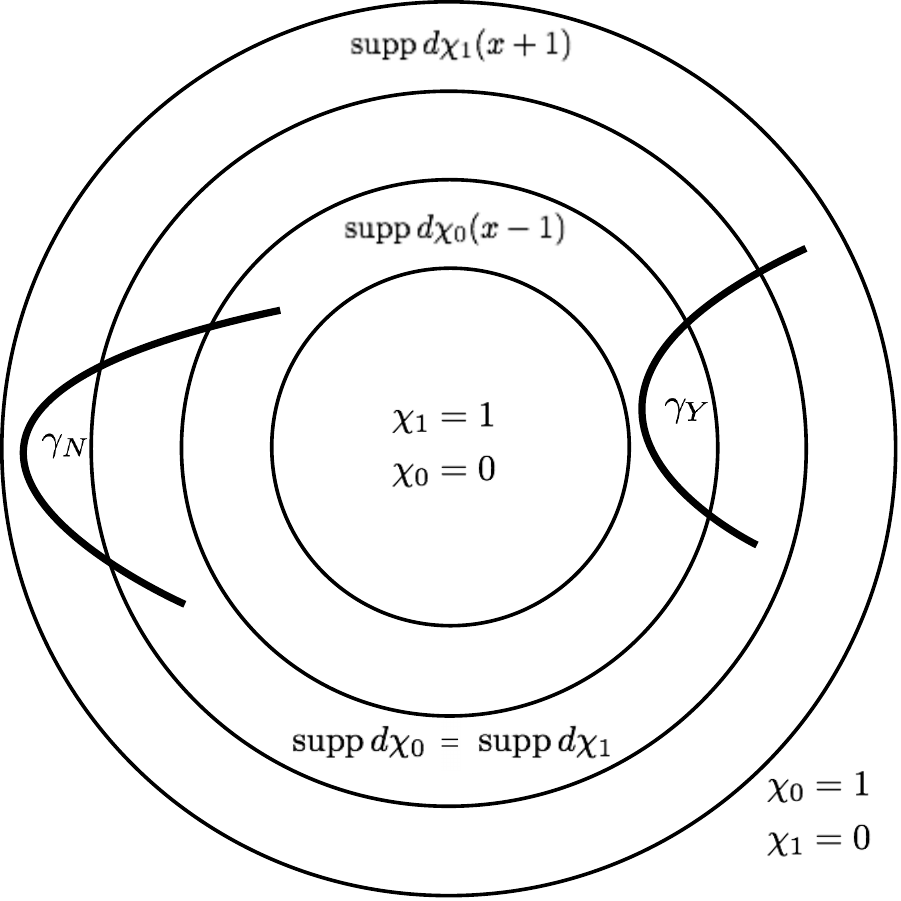}
\caption{The concentric circles indicate integer level sets of $x$: the outermost one is $x=1$ and the innermost $x=4$. The supports of the various cutoffs are indicated (note that the supports are contained in the interiors of their respective annuli). The trajectory $\gamma_N$ is ruled out by the convexity assumption \eqref{eq:convexity}, and this is exploited by Lemma \ref{zigzag:worse}. The trajectory $\gamma_Y$ is possible, and this is the reason a third iteration is needed in the parametrix construction.}\label{f:noreturn}
\end{figure}

Before proving this lemma we show how \eqref{e:a0a1} implies Theorem \ref{t:main}. We solve away the first error by writing, using \eqref{e:a2}
\[P(F - F A_0) = \Id + A_1 - A_0 A_0 - A_1 A_0 = \Id + A_1 - A_1 A_0.\]
Similarly we have
\[P(F - F A_0 - F A_1) = \Id - A_1 A_0 - A_0 A_1.\]

The last term is already $\Oh(h^\infty)$ by \eqref{e:a0a1}, but $A_1A_0$ is not yet small. We thus repeat this process for $A_1 A_0$ to obtain
\[\begin{split}
P(F - F A_0 - F A_1 + F A_1 A_0) &= \Id - A_0 A_1 + A_0 A_1 A_0 + A_1 A_1 A_0\\
&= \Id - A_0 A_1 + A_0 A_1 A_0 .
\end{split}\]

We now observe that both remaining error terms are of size $\Oh(h^\infty)$ thanks to \eqref{e:a0a1}.
Correspondingly, $\Id - A_0 A_1 + A_0 A_1 A_0$ is invertible for sufficiently small $h$, and the inverse is of the form $\Id+E$,
with $E=\Oh(h^\infty)$.
To estimate the resolvent we write out
\[F - F A_0 - F A_1 + F A_1 A_0 = F - \chi_1(x+1) R_1(\lambda)\chi_1 A_0  + \chi_0(x-1) R_0(\lambda)\chi_0 (-A_1+ A_1 A_0).\]
We then find that
\[\|\rho R(\lambda)\rho \| \le C(a_0 + a_1  + 2 h a_0a_1  + h^2 a_0^2a_1) \le Ch^2 a_0^2a_1.\]
This completes the proof that Lemma~\ref{zigzag:worse} implies Theorem~\ref{t:main}. Note that only $\rho_0$, the weight for $R_0$, and not $\rho_1$, appears in the definition of $\rho$. This is because $R_1(\lambda)$ is already multiplied by a compactly supported cutoff in every place where it appears in our parametrix (but this is not the case for $R_0(\lambda)$).

Lemma \ref{zigzag:worse} follows from the following two lemmas, for the hypotheses
(1)-(3) of Lemma~\ref{zigzag} are satisfied by $\varphi_j$ as in
Lemma~\ref{zigzag:worse},
and (4) follows from the support properties of $\varphi_j$ and
Lemma~\ref{lemma:bich-convex}.

\begin{lem}\label{zigzag} Suppose that $\varphi_1,\varphi_2,\varphi_3$ are
semiclassical differential operators with the properties that
\begin{enumerate}
\item
$\varphi_1$ is supported in $X_1$,
\item
$\varphi_2,\varphi_3$ are supported in $X_0\cap X_1$,
\item
$\supp\varphi_3\cap\supp\varphi_2=\emptyset$,
and $\supp\varphi_2\cap\supp\varphi_1=\emptyset$,
\item
there is no bicharacteristic of $P_1$ from a point $q_1 \in T^*(\supp \varphi_1 \cup(X_1\setminus X_0))$
to a point $q_2 \in T^*\supp \varphi_2 $ followed by a bicharacteristic of
$P_0$ from $q_2$ to a point $q_3 \in T^*\supp \varphi_3 $.
\end{enumerate}
Then
\[\|\varphi_3 R_0(\lambda) \varphi_2 R_1(\lambda) \varphi_1 \|_{L^2 \to L^2} = \Oh(h^\infty),\]
\end{lem}

\begin{lem}\label{lemma:bich-convex}
There is no bicharacteristic of $P_1$ from a point $q_1 \in T^*\{x>2\}$ to a point $q_2 \in T^*\{x<2\}$ followed by a bicharacteristic of $P_0$ from $q_2$ to a point $q_3 \in T^*\{x>2\}$.
\end{lem}

\begin{proof}[Proof of Lemma \ref{lemma:bich-convex}]
We prove this first in the case where the two curves constitute a bicharacteristic of $P$. If there were such a bicharacteristic, say $\gamma \colon [t_0,t_1]\to T^*X$, with $x(\gamma(t_0)),x(\gamma(t_1))>2$, and
$x(\gamma(\tau))<\min(x(\gamma(t_0)),x(\gamma(t_1)))$ for some
$\tau\in(t_0,t_1)$, then the function $x\circ\gamma$ would attain its minimum in the interior of $(t_0,t_1)$ at some point (and would be $<2$ there),
and the second derivative would be nonnegative there, contradicting our convexity assumption \eqref{eq:convexity}.

We now reduce to this case. Assume that there are curves,
$\gamma_0:[t_2,t_3]\to T^*X'_0$
a bicharacteristic of $P_0$ from $q_2$ to $q_3$ and
$\gamma_1:[t_1,t_2]\to T^*X'_1$
a bicharacteristic of $P_1$ from $q_1$ to $q_2$.
Now, by the
bicharacteristic convexity of $X_1$ in $X_1'$, $\gamma_1$ is completely in $X_1$
(since its endpoints are there), so it is a $P$ bicharacteristic. On the other hand,
$\gamma_0$
need not be a $P$ bicharacteristic since it might intersect $T^*(X_1\setminus X_0)$.
However, taking infimum $t_3'$
of times $t$ at which $x(\gamma(t))\geq x(q_3)$, $\gamma_0|_{[t_2,t'_3]}$ is
a $P$ bicharacteristic since it is disjoint from $T^*\{x>x(q_3)\}$ in view
of $x(q_2)<2$ and the intermediate value theorem. Thus, $\gamma:[t_1,t_3']\to T^*X$
given by $\gamma_1$ on $[t_1,t_2]$ and $\gamma_0$ on $[t_2,t_3']$ is
a $P$ bicharacteristic, with $x(\gamma(t_1))>2$, $x(\gamma(t_2))<2$,
$x(\gamma(t_3'))>2$, completing the reduction to the case in the previous paragraph.
\end{proof}

\begin{proof}[Proof of Lemma \ref{zigzag}]
First suppose that $u \in L^2(X)$ is polynomially bounded; we claim that
\begin{equation}\label{eq:pointwise-bd}
\|\varphi_3 R_0(\lambda) \varphi_2 R_1(\lambda) \varphi_1 u\|_{L^2} = \Oh(h^\infty).
\end{equation}
For this, it suffices to show that
$\WFh(\varphi_3 R_0(\lambda) \varphi_2 R_1(\lambda) \varphi_1 u)=\emptyset$.
Note that by the polynomial boundedness assumption on the resolvent,
$\varphi_3 R_0(\lambda) \varphi_2 R_1(\lambda) \varphi_1 u$, as well as
$\varphi_2 R_1(\lambda) \varphi_1 u$, are polynomially bounded.

So suppose $q_3\in \WFh(\varphi_3 R_0(\lambda) \varphi_2 R_1(\lambda) \varphi_1 u)$,
so in particular $q_3\in T^*\supp\varphi_3\cup S^*\supp\varphi_3$
and, as $\varphi_3$ is microlocal,
$q_3\in \WFh(R_0(\lambda) \varphi_2 R_1(\lambda) \varphi_1 u)$.
Now, if $q_3$ is not in the characteristic set of $P_0$, then by microlocal
ellipticity of $P_0$, $q_3\in \WFh(\varphi_2 R_1(\lambda) \varphi_1 u)$, thus in
$T^*\supp\varphi_2\cup S^*\supp\varphi_2$. This contradicts (3).

So we may assume that $q_3$ in the characteristic set of $P_0$ (and hence in particular not in $S^*X$).
By (0-OG), noting that $\varphi_2$ and $\varphi_3$ have disjoint supports,
there is a point
$q_2\in\WFh(\varphi_2 R_1(\lambda) \varphi_1 u)$ on the backward $P_0$-bicharacteristic from $q_3$. Thus $q_2\in T^*\supp\varphi_2$
and $q_2\in \WFh(R_1(\lambda) \varphi_1 u)$. By (1-OG), noting that
$\varphi_1$ and $\varphi_2$ have disjoint supports, either
the backward $P_1$ bicharacteristic from $q_2$
intersects $T^*(X_1\setminus X_0)$, in which case we can take any $q_1$ on it
in this region, or there is a point $q_1$ on this backward bicharacteristic
in $\WFh(\varphi_1 u)$, which is thus in $T^*\supp\varphi_1$. Since this
contradicts (4), it completes
the proof of \eqref{eq:pointwise-bd}.

To complete the proof of the lemma, we just note that for any $N$,
the family of
operators
$$
h^{-N}\varphi_3 R_0(\lambda) \varphi_2 R_1(\lambda) \varphi_1,
$$
dependent on $h$ and $\lambda$, is continuous on $L^2$, and for each $u$,
$h^{-N}\varphi_3 R_0(\lambda) \varphi_2 R_1(\lambda) \varphi_1 u$ is uniformly
bounded in $L^2$. Thus, by the theorem of Banach-Steinhaus,
$h^{-N}\varphi_3 R_0(\lambda) \varphi_2 R_1(\lambda)\varphi_1$ is uniformly bounded (in $h$ and $\lambda$)
on $L^2$, completing the proof of the lemma.
\end{proof}

\begin{rem}
The application of Banach-Steinhaus is only needed because we merely made wavefront
set assumptions in Definition~\ref{Def:outgoing}. In practice, the wave front set
statement is proved by means of a uniform estimate, and thus
Banach-Steinhaus is superfluous.
\end{rem}

\begin{rem} Lemma \ref{zigzag:worse} holds with the same proof if $\varphi_j$ are instead semiclassical pseudodifferential operators with $\WFh \varphi_j$ in the cotangent bundle of the corresponding set. Note however that in this case slightly more care is needed in defining the $\varphi_j$ since their Schwartz kernels may no longer be compactly supported. This could be useful for applications where $P$ is not differential, as in \cite{sz}.
\end{rem}

\section{Model operators near infinity}\label{s:infinity}
In this section we describe some examples in which the assumptions on the
model at infinity, $P_0$, are satisfied.
Recall that the assumptions on $P_0$ are of three kinds:
\begin{enumerate}
\item
bicharacteristic convexity of level sets of $x$ for $0<x<4$,
\item
polynomial bounds for the cutoff resolvent,
\item
semiclassically outgoing resolvent.
\end{enumerate}

For simplicity, in this section we consider the case
\[
P_0 = -h^2\Delta_g - 1.
\]
We start with some general remarks.

First, in the setting where $X_0'$ is diffeomorphic to $\RR^n$,
has nonpositive sectional curvature and,
for fixed $z_0$ the function $x(z)=F(d(z,z_0))$ with $F'<0$,
\eqref{eq:convexity} follows from the Hessian comparison theorem \cite{sy, vw}.

Next, the semiclassically outgoing assumption is satisfied for
$R_0(\lambda)$ if the restriction of its Schwartz kernel to
$(X_1\cap X_0)^2\setminus\Diag$
is a semiclassical Fourier integral operator with canonical relation $\Lambda'$
corresponding
to forward propagation along bicharacteristics, i.e.\ $(y,z,\eta,\zeta)\in\Lambda'$
implies $(y,\eta)$ is on the forward bicharacteristic segment from $(z,\zeta)$.
Here $\Diag$ is the diagonal in $(X_1\cap X_0)^2$. Note that this is where
restricting the semiclassical outgoing condition to its off-diagonal version is
useful, in that usually the structure of the resolvent at the diagonal is slightly
more complicated (though the condition would still hold); see also
Remark~\ref{rem:outgoing}.

\subsection{Asymptotically Euclidean manifolds}\label{s:aeinfinity}
If $X$ is isometric outside of a compact set to Euclidean space we may take
$X_0=\RR^n$ with the Euclidean metric $g_0$, and $x^{-1}$
the distance function from a point in $\RR^n$. Thus, the convexity hypotheses
\eqref{eq:convexity} holds in view of geodesic convexity of the spheres. Moreover,
for $\Gamma_1>\Gamma>0$, $\lambda_0>0$,
the resolvent continues analytically to
$\{\lambda:\ \im\lambda<\Gamma h,\ \re\lambda>\lambda_0\}$ as an operator
$$
R(\lambda):e^{-\Gamma_1|z|} L^2\to e^{\Gamma_1 |z|} L^2
$$
with uniform estimates $\|R(\lambda)\|\leq Ch^{-1}$. Finally, $R(\lambda)$ is
a semiclassical FIO associated to the forward flow; indeed, with $\sqrt{}$ the square
root on $\C \setminus(-\infty,0]$ which is positive for positive $\lambda$,
its Schwartz kernel is (see e.g. \cite{mel})
$$
R(\lambda,y,z)=(h^{-1}\sqrt \lambda)^{n-2} e^{i\sqrt{\lambda}|y-z|/h}a(\sqrt{\lambda}|y-z|/h),
$$
where $a$ is a symbol (away from the origin).

The applications in this case have already been treated in \cite{Nonnenmacher-Zworski:Quantum, wz}, but for compactly supported cutoff functions. The novelty  in the present paper in this setting is that we use exponential weights $e^{-\Gamma_j|z|}$. More general asymptotically Euclidean manifolds, whose metrics have holomorphic coefficients near infinity, could probably also be treated: see \cite{wz0,wz} for more details on the needed assumptions and the proof of the analytic continuation of the resolvent, and \cite{vz,d} for semiclassical estimates and propagation of singularities. 

\subsection{Asymptotically hyperbolic manifolds}\label{s:ahinfinity}
The convexity assumption \eqref{eq:convexity} is satisfied for the geodesic flow on a general asymptotically hyperbolic metric. In the following lemma this is proved in a region $\{x < \eps\}$, but a rescaling of the boundary defining function gives it in the region $\{x < 4\}$. The computation is standard, but we include it for the reader's convenience.

\begin{lem}
Let $x$ be a boundary defining function on $\overline X$, a compact manifold with boundary, and let $g$ be a metric on the interior of the form
\[g = \frac{dx^2 + \tilde g}{x^2}, \textrm{ near } \D \overline{X}\]
where $\tilde g$ is a family of metrics on $\D \overline X$, smooth up to $\D \overline X$. Then for $x$ sufficiently small we have
\[\dot x(t)=0\Rightarrow \ddot x(t)<0\]
along geodesic bicharacteristics.
\end{lem}

As remarked in the introduction, it is possible to reduce a more general form of the metric $g$ to this one.  Namely, it suffices to assume that $\tilde g$ is a 2-cotensor which is a metric on $\D\overline{X}$ when restricted to $\D \overline{X}$: see \cite[Proposition 2.1]{js}.

\begin{proof}
If $(x,y)$ are coordinates on $X$ near $\D \overline X$ such that $y$ are coordinates on $\D \overline X$, and if $\xi$ is dual to $x$ and $\eta$ to $y$, then the geodesic Hamiltonian is given by
\[|\zeta|^2 = \tau^2 + \tilde g^{-1}(\mu,\mu),\]
where $\tau = x \xi$ and $\mu = x \eta$, and $\tilde g^{-1}$ is the bilinear form on $T^*\D \overline X$ induced by $\tilde g$.
Its Hamiltonian vector field is
\[ H_{|\zeta|^2} = \D_\xi |\zeta|^2 \D_x - \D_x |\zeta|^2 \D_\xi + (\D_\eta|\zeta|^2) \cdot \D_y - (\D_y|\zeta|^2)\cdot \D_\eta.\]
We use $\D_\xi = x \D_\tau$, $\D_\eta = x \D_\mu$ and ``$\D_x = \D_x + x^{-1}\mu \cdot \D_\mu + x^{-1}\tau\D_\tau$'', where in the last formula the left hand side refers to $(x,y,\xi,\eta)$ coordinates, and the right hand side to $(x,y,\tau,\mu)$ coordinates. This gives
\[\begin{split}H_{|\zeta|^2} &= x \D_\tau |\zeta|^2(\D_x + x^{-1} \mu \cdot \D_\mu + x^{-1} \tau\D_\tau) \\
&\quad- \left[\left(x\D_x + \mu \cdot \D_\mu + \tau \D_\tau\right)|\zeta|^2\right]\D_\tau + x(\D_\mu|\zeta|^2) \cdot \D_y - x(\D_y|\zeta|^2)\cdot \D_\mu.\end{split}\]
We cancel the $\D_\tau(|\zeta|^2)\tau \D_\tau$ terms, write $H_{\tilde g} = (\D_\mu|\zeta|^2) \cdot \D_y - (\D_y|\zeta|^2)\cdot \D_\mu$, substitute $|\zeta|^2 = \tau^2 + \tilde g(\mu,\mu)$, and use $\mu \cdot \D_\mu \tilde g^{-1}(\mu,\mu) = 2\tilde g^{-1}(\mu,\mu)$. Now
\[H_{|\zeta|^2} = 2\tau x\D_x + 2\tau \mu \cdot \D_\mu - (2\tilde g^{-1}(\mu,\mu) + x \D_x \tilde g^{-1}(\mu,\mu))\D_\tau + xH_{\tilde g}.\]
We now observe from this that, along flowlines of $H_{|\zeta|^2}$, we have $\dot x = 2 \tau x$ and $\dot \tau = - 2 \tilde g^{-1}(\mu,\mu) - x \D_x \tilde g^{-1}(\mu,\mu)$. Hence
\[\dot x(t)=0\Rightarrow \tau = 0,\]
in which case 
\[\ddot x = -4x \tilde g - 2 x^2 \D_x \tilde g^{-1}.\]
Since $\tilde g^{-1}|_{x=0}$ is positive definite, for sufficiently small $x$ this is always negative.
\end{proof}

If in addition $\tilde g$ is even in $x$, in the sense that the Taylor series at $x=0$ includes only even powers of $x$ (or see \cite[Definition 1.2]{g} for a more invariantly phrased version of this condition), then work of the second author \cite[Theorem 4.3]{v1}, \cite[Theorem 5.1]{v2} implies the polynomial bound \eqref{eq:model-bound} and the outgoing condition (0-OG) for
\[
P_0 = h^2\Delta_g - 1,
\]
when the manifold is nontrapping.
The outgoing condition which is proved in those theorems, when restricted to data in $L^2(X_0 \cap X_1)$, is the same as that in condition (0-OG) (for this purpose the weights are irrelevant). The resolvent estimate \cite[(4.27)]{v2} is that with
$\cR(\sigma)=(\Delta_g-(\frac{n-1}{2})^2-\sigma^2)^{-1}$ for
$\re\sigma\gg 0$,
for
$s>\frac{1}{2}+\Gamma/2$ and $-\Gamma/2<\im\sigma$,
\begin{equation}\label{e:vasyresest0}
\|x^{-(n-1)/2+i\sigma} \cR(\sigma)f\|_{H^s_{|\sigma|^{-1}}(\olX_{\even})}
\leq C|\sigma|^{-1}\|x^{-(n+3)/2+i\sigma}f\|_{H^{s-1}_{|\sigma|^{-1}}(\olX_{\even})}.
\end{equation}
We will show that this implies
\begin{equation}\label{e:vasyresest}
\|x^{1+\Gamma/2} R_0(\lambda) x^{5/2 + \Gamma/2}\|_{L^2_g(X) \to L^2_g(X)}\le C/h,
\end{equation}
uniformly for $\re \lambda \in [-E,E]$, $\im \lambda > -\Gamma$. This
argument is somewhat involved due to the rather different functions
spaces appearing in \eqref{e:vasyresest0} and
\eqref{e:vasyresest}, as already indicated by the presence of $\olX_{\even}$ in
\eqref{e:vasyresest0}; the results of \cite[Theorem 4.3]{v1},
\cite[Theorem 5.1]{v2} are obtained by extending an operator related
to the spectral family of the Laplacian across $\pa \olX_{\even}$ to a
larger space. Here $\olX_{\even}$
is $\olX$ as a topological manifold with boundary, but with smooth
structure given by even (in $x$) smooth functions on $\olX$; effectively
this means that the boundary defining function $x$ is replaced by
$\mu=x^2$.

\begin{proof}[Proof that \eqref{e:vasyresest0} implies
  \eqref{e:vasyresest}] We first recall the definition of
  $H^s_{|\sigma|^{-1}}(\olX_{\even})$, which is the standard
  semiclassical (with $|\sigma|^{-1}$ playing the role of the
  semiclassical parameter) Sobolev space on $\olX_{\even}$. A straightforward computation gives, see \cite[Section~1]{v2},
\begin{equation}\label{eq:L2-rel}
\|x^{-(n+1)/2}u\|_{L^2(\olX_{\even})}\sim \|u\|_{L^2_{g}(X)},
\end{equation}
where $L^2(\olX_{\even})$ is with respect to any smooth non-degenerate density on the
compact manifold $\olX_{\even}$), while $L^2_{g}(X)$ is the metric
$L^2$-space.
Furthermore, in local coordinates $(\mu,y)$,
using $2\pa_\mu=x^{-1}\pa_x$, for $l\geq 0$ integer, the squared
high energy $H^l_{|\sigma|^{-1}}(\olX_{\even})$
norm of $u$ is equivalent to
$$
\sum_{k+|\alpha|\leq l}\||\sigma|^{-k-|\alpha|}(x^{-1}\pa_x)^k \pa_y^\alpha u\|_{L^2(\olX_{\even})}^2.
$$

We now convert \eqref{e:vasyresest0}  into an $H^s_0(\olX)$-estimate, where $H^s_0(\olX)$ are
the zero-Sobolev spaces of Mazzeo and Melrose \cite{mm}, i.e.\ they
are the Sobolev spaces measuring regularity with respect to $\Diff_0(\olX)$, the algebra of
differential operators generated by $\Vf_0(\olX)$, the Lie algebra of
$\CI$ vector fields vanishing at
the boundary over $\CI(\olX)$, in the space $L^2_g(X)$. More
precisely, we need the semiclassical version $H^s_{0,|\sigma|^{-1}}(\olX)$ of
these spaces, in which
$|\sigma|^{-1}$ times $\Vf_0(\olX)$ is used to generate the differential operators. The square high energy $H^l_{0,|\sigma|^{-1}}(\overline X)$ norm of $u$ is equivalent to
\[
\sum_{k+|\alpha|\leq l}\||\sigma|^{-k-|\alpha|}(x\pa_x)^k (x\pa_y)^\alpha u\|_{L^2_g(X)}^2.
\]
Because of the ellipticity of $\Delta_g$ for these spaces (see \cite{Melrose-SaBarreto-Vasy:Semiclassical}), in the
precise sense that the standard principal symbol is elliptic even in
the semiclassical zero-calculus, this is also equivalent, when $l$ is even, to
\begin{equation}\label{e:0ellip}
\| u\|_{L^2_g(X)}^2 + \||\sigma|^{-l} \Delta_g^{l/2} u\|_{L^2_g(X)}^2.
\end{equation}
This equivalence identifies the $H^l_{0,|\sigma|^{-1}}$ spaces with the usual semiclassical Sobolev spaces based on $L^2_g(X)$.

To make the conversion from \eqref{e:vasyresest0} into an $H^s_{0,|\sigma|^{-1}}(\olX)$-estimate, we remark
that with $k+|\alpha|\leq l$,
$$
(x^{-1}\pa_x)^k \pa_y^\alpha \in
x^{-2k-|\alpha|}\Diff_0^l(\olX)\subset x^{-2l}\Diff_0^l(\olX),
$$
and similarly for the high energy spaces, so
$$
\|u\|_{H^l_{|\sigma|^{-1}}(\olX_{\even})}\lesssim \|x^{(n+1)/2-2l}u\|_{H^l_{0,|\sigma|^{-1}}(\olX)},
$$
where the shift of $(n+1)/2$ in the exponent is due to the different
normalization of the $L^2$-spaces, \eqref{eq:L2-rel}.
Thus, taking $s\geq 1$ integer, $s>1/2+\Gamma/2$, $\im\sigma>-\Gamma/2$, $\Gamma>0$, and simply using
$\|u\|_{L^2(\olX_{\even})} \le \|u\|_{H^s_{|\sigma|^{-1}}(\olX_{\even})}$, we deduce that
\begin{equation*}\begin{split}
\|x^{1-\im\sigma}\cR(\sigma)f\|_{L^2_g(X)}
&\le C \|x^{-(n-1)/2+i\sigma}\cR(\sigma)f\|_{L^2(\olX_{\even})}
\le C \|x^{-(n-1)/2+i\sigma}
\cR(\sigma)f\|_{H^s_{|\sigma|^{-1}}(\olX_{\even})}\\
&\le C
|\sigma|^{-1}\|x^{-(n+3)/2+i\sigma}f\|_{H^{s-1}_{|\sigma|^{-1}}(\olX_{\even})}
\le C
|\sigma|^{-1}\|x^{-2s+1+i\sigma}f\|_{H^{s-1}_{0,|\sigma|^{-1}}(\olX)} \\
&\le C
|\sigma|^{-1}\|x^{-2s+1-\im\sigma}f\|_{H^{s-1}_{0,|\sigma|^{-1}}(\olX)}.
\end{split}\end{equation*}
Notice that there is a loss of $x^{-2s}$ in the weight between the two
sides. Although this simple argument does not give an optimal
zero-Sobolev space estimate, to minimize losses take $3/2+\Gamma/2\geq s$,
and $-\Gamma/2<\im\sigma<-\Gamma/2+1/2$, so
$-2s+1-\im\sigma>-5/2-\Gamma/2$
\begin{equation}\label{e:vasyintermediate}\begin{split}
\|x^{1+\Gamma/2}\cR(\sigma)f\|_{L^2_g(X)}
&\le C \|x^{1-\im\sigma}\cR(\sigma)f\|_{L^2_g(X)}
\le C
|\sigma|^{-1}\|x^{-2s+1+i\sigma}f\|_{H^{s-1}_{0,|\sigma|^{-1}}(\olX)}\\
&\le C
|\sigma|^{-1}\|x^{-5/2-\Gamma/2}f\|_{H^{s-1}_{0,|\sigma|^{-1}}(\olX)}.
\end{split}
\end{equation}
Again using the ellipticity of $\Delta_g$ in the zero-calculus (as in \eqref{e:0ellip}), allows one to strengthen the norm on
the left hand side to
$\|x^{1+\Gamma/2}\cR(\sigma)f\|_{H^{s+1}_{0,|\sigma|^{-1}}(\olX)}$ -- this
simply requires a commutator argument or a parametrix with a smoothing (but not semiclassically
trivial) error. For example, to strengthen the norm to $\|x^{1+\Gamma/2}\cR(\sigma)f\|_{H^2_{0,|\sigma|^{-1}}(\olX)}$ we may write
\begin{equation}\label{e:resident}
\Delta_g x^{1+\Gamma/2}\cR(\sigma)f = x^{1+\Gamma/2} f + \left(\frac{(n-1)^2}4 + \sigma^2 \right)x^{1+\Gamma/2}\cR(\sigma)f + [\Delta_g, x^{1+\Gamma/2}]\cR(\sigma)f.
\end{equation}
Multiplying by $|\sigma|^{-2}$ and taking the $L^2_g(X)$ norm, we see that the first two terms are both controlled by the estimate \eqref{e:vasyintermediate}, while the last is bounded by 
\[
|\sigma|^{-1}\,\|x^{1+\Gamma/2}\cR(\sigma)f\|_{H^1_{0,|\sigma|^{-1}}(\olX)} \le \frac 12 \|x^{1+\Gamma/2}\cR(\sigma)f\|_{H^2_{0,|\sigma|^{-1}}(\olX)}.
\]
This implies that \eqref{e:vasyintermediate} holds with $L^2_g(X)$ norms replaced by $H^2_{0,|\sigma|^{-1}}$ norms, and iterating one can get as far as controlling the $H^{s+1}_{0,|\sigma|^{-1}}(\olX)$ norm (past which point the first term of \eqref{e:resident} is no longer controlled).

To pass from estimates in $H^{s+1}_{0,|\sigma|^{-1}}(\olX)$ to estimates in $L^2_g(X)$ we use a similar procedure. For $K>0$ fixed we have  the semiclassical elliptic estimate  
\begin{equation}\label{e:0ellipest}
\|u\|_{H^{s-1}_{0,|\sigma|^{-1}}}  \le C|\sigma|^{-2}\|x^a (\Delta_g + K^2|\sigma|^2)x^{-a}u \|_{H^{s-3}_{0,|\sigma|^{-1}}},
\end{equation}
 from which we deduce
\[\begin{split}
\|x^{1+ \Gamma/2}&\cR(\sigma)f\|_{H^{s-1}_{0,|\sigma|^{-1}}} \\&\le \|x^{1+ \Gamma/2}\cR(iK|\sigma|)f\|_{H^{s-1}_{0,|\sigma|^{-1}}} + |\sigma|^2(1 + K^2) \|x^{1+ \Gamma/2}\cR(\sigma)\cR(iK|\sigma|)f\|_{H^{s-1}_{0,|\sigma|^{-1}}} \\
&\le C |\sigma|^{-2}\|x^{1+ \Gamma/2}f\|_{H^{s-3}_{0,|\sigma|^{-1}}} +C|\sigma|\|x^{-5/2 - \Gamma/2}\cR(iK|\sigma|)f\|_{H^{s-1}_{0,|\sigma|^{-1}}} \\
&\le C |\sigma|^{-1}\|x^{-5/2 - \Gamma/2}f\|_{H^{s-3}_{0,|\sigma|^{-1}}}.
\end{split}\]
Note that although the resolvent we use is $\cR(\sigma) = (\Delta_g - (n-1)^2/4 - \sigma^2)^{-1}$, the shift by $(n-1)^2/4$ is not important when $|\sigma|$ is large and does not interfere with the application of \eqref{e:0ellipest}.
Iterating this
we obtain the more standardly phrased weighted estimate
\begin{equation*}
\|x^{1+\Gamma/2}\cR(\sigma)f\|_{H^{2}_{0,|\sigma|^{-1}}(\olX)}\le C |\sigma|^{-1}\|x^{-5/2-\Gamma/2}f\|_{L^2_g(X)},
\end{equation*}
which in turn implies \eqref{e:vasyresest}.
This is not optimal in terms of the weights which could be improved
using $\olX_{\even}$ derivatives to estimate weights in the spirit of
\cite[Chapter 4, Lemma~5.4]{taylor}, but this would result in a loss
in terms of $|\sigma|$.
\end{proof}

Another approach to obtaining polynomial boundedness of the resolvent and the semiclassical outgoing condition is possible in a special case. Let $X'_0=\BB^n$ with a metric which is asymptotically hyperbolic
in the following stronger sense:
\begin{equation}\label{e:ah}
g_0= g_{\HH^n}+\chi_\delta(x)\tilde g,
\ P_0=h^2\left(\Delta_{g_0}+x^2V_0-\frac{(n-1)^2}{4}\right)-\lambda_0,\ \lambda_0>0,
\end{equation}
where $g_{\HH^n}$ is the hyperbolic metric on $\BB^n$ and $\tilde g$ is a smooth
symmetric 2-cotensor on $\overline{\BB^n}$, $V_0\in\CI(\overline{\BB^n})$,
 $\chi_\delta(t)=\chi(t/\delta)$, $\chi\in\CI(\RR)$ supported
in $[0,1)$, identically $1$ near $0$, and $\delta>0$ sufficiently small.
This is the setting considered by
Melrose, S\'a Barreto
and the second author \cite{Melrose-SaBarreto-Vasy:Semiclassical}.
Note that, although we have $g_0=g_{\HH^n}$ on a large compact set,
the factor $\chi_\delta$ does not change $g_0$ near infinity. Thus, after possibly
scaling $x$, i.e.\ replacing it by $x/\eps$, in the region $x<4$ the cutoff
$\chi_\delta\equiv 1$.

It is shown in \cite{Melrose-SaBarreto-Vasy:Semiclassical} that the Schwartz kernel
of $(P_0-\lambda)^{-1}$ is
a semiclassical paired Lagrangian distribution, which is just a Lagrangian
distribution away from the diagonal associated to the flow-out of the diagonal by
the Hamilton vector field of the metric function, hence, as remarked at the beginning
of the section, $(P_0-\lambda)^{-1}$ is semiclassically outgoing. This also gives that
$(P_0-\lambda)^{-1}$ satisfies the bound in \eqref{eq:model-bound}
with $D = [-E,E]-i[0,\Gamma h]$ and with $a_0 = C/h$ for
arbitrary $\Gamma > 0, E \in (-\lambda_0,\lambda_0)$ with compactly supported cutoffs as a
consequence of a semiclassical version of \cite[Theorem 3.3]{gu}.
Moreover, it is also shown in \cite{Melrose-SaBarreto-Vasy:Semiclassical} that
the resolvent satisfies weaker polynomial bounds in weighted spaces, namely
$R_j(\lambda):x^a L^2\to x^{-b}L^2$,
$a,b>C$, with $a_0(h)=C'h^{-1 - (n-1)/2}$. It is highly likely that the better bound $a_0 = C/h$ could be shown for the weighted spaces in this way as well; this could be proved by extending
the approach of \cite{gu} in a manner that is uniform up to the boundary (i.e.\ infinity);
this is expected to be relatively straightforward. 
The same results hold without modification in the case where $X'_0$ is a disjoint union of balls with $g_0$ and $P_0$ of the form \eqref{e:ah} in each ball.

\section{Model operators for the trapped set}\label{s:trapped}
In this section we describe some examples in which the assumptions on the
model near the trapped set, $P_1$, are satisfied. The two main assumptions, polynomial boundedness of the resolvent \eqref{eq:model-bound} and the semiclassically outgoing property (1-OG), are the same as in the case of $P_0$ above, with the exception that the latter need only hold at points where the backward bicharacteristic is disjoint from any trapping in $X_1$ . 

In \S\ref{s:cpx} we prove that the semiclassically outgoing property
(1-OG) holds for polynomially bounded resolvents when either a complex
absorbing barrier is added near infinity (regardless of the cutoff or
weight and regardless of the type of infinite end), and in
\S\ref{s:euc} we prove it when infinity is Euclidean (with no complex
absorption added) and the resolvent is polynomially bounded and suitably cutoff or weighted. In \S\ref{s:wz}, \S\ref{s:nz}, \S\ref{s:ps} we give examples of assumptions on the trapped set which imply that the resolvent is polynomially bounded.

\subsection{Complex absorbing barriers}\label{s:cpx}
In this subsection we consider model operators of the form
\begin{equation}\label{e:abs}
P_1 = h^2\Delta_g + V(x) - i W,
\end{equation}
where $V \in C^\infty(X_1',\R)$ and $W \in C^\infty(X_1';[0,1])$ has $W = 0$ on $X_1$ and $W = 1$ off a compact set. Suppose that each backward bicharacteristic of $h^2\Delta_g + V(x)$ at energy $\lambda \in [-E,E]$ enters either the interior of $T^*\left[(X_1 \setminus X_0) \cup W^{-1}(1)\right]$ in finite time. The strong assumptions on $W$ remove the need for any further assumptions on $V$ or on the metric.

The function $W$ in \eqref{e:abs} is called a complex absorbing barrier and serves to suppress the effects of infinity. In Lemma \ref{l:w} we prove the needed semiclassical propagation of singularities in this setting, that is to say that $R_1(\lambda)$ is semiclassically outgoing in the sense of \S\ref{s:abstract}. After this, all that is needed to be in the setting of \S\ref{s:abstract} is the convexity condition \eqref{eq:convexity} and the resolvent estimate \eqref{eq:model-bound}. In \S\ref{s:wz} and  \S\ref{s:nz} we describe settings in which results of Wunsch-Zworski \cite{wz} and Nonnenmacher-Zworski \cite{Nonnenmacher-Zworski:Quantum, nz2} respectively give the needed bound \eqref{eq:model-bound}.

For the following lemma we use a positive commutator argument based on an escape function as in \cite{vz}, which is the semiclassical adaptation of the proof of \cite[Proposition 3.5.1]{h}. The only slight subtlety comes from the interaction of the escape function with the complex absorbing barrier $W$ and from the possibly unfavorable sign of $\im \lambda$, but the positive commutator with the self adjoint part of the operator overcomes these effects. See also \cite[Lemma A.2]{Nonnenmacher-Zworski:Quantum} for a similar result.

\begin{lem}\label{l:w}
Suppose that $P_1$ is as in \eqref{e:abs}. Let $U \subset T^*(X_1'\setminus (X\setminus X_0))$ be preserved by the backward bicharacteristic flow. If $u = u_h \in L^2(X_0\cap X_1)$ has $\|u\|_{L^2} = 1$ and
\begin{equation}\label{eq:lemabs0}\|\Op(a)(P_1 - \lambda)u\|_{L^2} = \Oh(h^\infty)\end{equation}
for all $a \in C_0^\infty(T^*X'_1)$ with support in $U$ and all $\lambda \in [-E,E]-i[0,\Gamma h]$, then for every $a \in C_0^\infty(T^*X'_1)$ with support in $U$ we have also
\begin{equation}\label{eq:lemabs}\|\Op(a) u\|_{L^2} = \Oh(h^\infty).\end{equation}
The implicit constants in $\Oh$  in \eqref{eq:lemabs0} and \eqref{eq:lemabs} are uniform for $\lambda \in [-E,E]-i[0,\Gamma h]$.
\end{lem}

Note that in view of Definition \ref{Def:outgoing}, this lemma implies assumption (0-OG)

\begin{proof}
In this proof all norms are $L^2$ norms. In the first step we use ellipticity to reduce to a neighborhood of the energy surface, and then a covering argument to reduce to a neighborhood of a single bicharacteristic segment. In the second step we construct an escape function (a monotonic function) along this segment. In the third step we implement the positive commutator method. Let
\[
p = |\xi|_g^2 + V(x), \qquad p_1 = p - iW.
\]

\textbf{Step 1.} Observe first that for any $\delta > 0$, we can find $R_\delta(\lambda)$, a semiclassical elliptic inverse for $P_1$ on the set $\{|p_1-\lambda|>\delta\}$, such that
\[\|\Op(a) u\| = \|\Op(a) R_\delta(\lambda)(P_1 - \lambda)u\| + \Oh(h^\infty)\]
as long as $\supp a \subset \{|p_1-\lambda|>\delta\}$. Since by the semiclassical composition formula the operator $\Op(a) R_\delta(\lambda)$ is the quantization of a compactly supported symbol with support contained in $\supp a$, plus an error of size $\Oh(h^\infty)$ (as an operator $L^2 \to L^2$), we have the lemma for $a$ with $\supp a \subset \{|p_1-\lambda|>\delta\}$. It remains to study $a$ with $\supp a \subset \{|p_1-\lambda|< 2\delta\}$. Note that this is a precompact set for $\delta$ small, because of the condition that $W = 1$ off of a compact set.

Now fix $a_0\in C_0^\infty(T^*X_1)$ for which we wish to prove \eqref{eq:lemabs}. Take $U_0$ with $\overline{U_0} \subset U$ such that $U_0$ is preserved by the backward flow and $\supp a_0 \subset U_0$. For each $\zeta \in \overline {U_0} \cap p_1^{-1}(\lambda)$ put 
\[T_\zeta \Def \sup\{t; t<0, \Phi^t\zeta \in T^*W^{-1}([2\eps,1])\},\]
where $\Phi^t$ is the flow of the Hamiltonian vector field of $p$ at time $t$, and $\eps > 0$ will be specified later. The supremum is taken over a nonempty set, since each backward bicharacteristic of $p$ was assumed to enter either $T^*(W^{-1}(1))$ or $T^*(X_1 \setminus X_0)$ in finite time, and the second possibility is ruled out by the assumption on $U$.
%otherwise we would have $\Phi^{t_0}\zeta \in T^*(X\setminus X_0)$ for some $t_0<0$, and consequently $\Phi^{t_0}\zeta \notin U_0$, contradicting the assumption that $U_0$ is preserved by the backward flow. 
We will prove the lemma for $a$ which are supported in a sufficiently small neighborhood $V_\zeta$ of $\{\Phi^t\zeta; T_\zeta \le t \le 0\}$. This gives the full lemma because if $\delta$ small enough these neighborhoods together with $\{|p_1-\lambda|>\delta\}$ cover all of $\supp a_0$.

\textbf{Step 2.} To do this we take a tubular neighborhood $U_\zeta \subset U_0$ of $\{\Phi^t\zeta; T_\zeta \le t \le 0\}$, that is a neighborhood of the form
\begin{equation}\label{eq:uzeta}U_\zeta = \bigcup_{-\eps_\zeta + T_\zeta < t < \eps_\zeta}\Phi^t(\Sigma_\zeta \cap U_\zeta),\end{equation}
where $\Sigma_\zeta \subset T^*X$ is a hypersurface transversal to the bicharacteristic through $\zeta$, and $U_\zeta$ and $\eps_\zeta$ are small enough that
\[\bigcup_{-\eps_\zeta + T_\zeta< t < T_\zeta +  \eps_\zeta}\Phi^t(\Sigma_\zeta \cap U_\zeta) \subset T^*W^{-1}([\eps,1]),\]
and also small enough that the map $U_\zeta \to (-\eps_\zeta + T_\zeta, \eps_\zeta) \times( \Sigma_\zeta \cap U_\zeta)$ defined by \eqref{eq:uzeta} is a diffeomorphism. We now use these `product coordinates' to define an escape function as follows. Take
\begin{itemize}
\item $\varphi_\zeta \in C_0^\infty(\Sigma_\zeta \cap U_\zeta;[0,1])$ with $\varphi_\zeta = 1$ near $\zeta$, and
\item $\chi_\zeta \in C_0^\infty((-\eps_\zeta + T_\zeta, \eps_\zeta);([0,\infty))$ with $\chi_\zeta' \le -1$ near $[T_\zeta,0]$ and $\chi_\zeta' \le -2\Gamma \chi_\zeta$ on $[T_\zeta,\eps_\zeta]$.
\end{itemize}
The constant $\Gamma$ above is the same as the one in the statement of the lemma. Put
\[q_\zeta \Def \varphi_\zeta \chi_\zeta \in C_0^\infty(T^*X_1'), \qquad \{p,q_\zeta\} = \varphi_\zeta \chi'_\zeta,\]
and let $V_\zeta$ be a neighborhood of $\{\Phi^t\zeta; T_\zeta \le t \le 0\}$ in which $\chi'_\zeta \le -1$ and $\chi'_\zeta \le -2\Gamma\chi_\zeta$.
Take $b \ge 0$ such that
\begin{equation}\label{e:bdef}
b^2 = -\{p,q_\zeta^2\} + r, \qquad r \in C_0^\infty(T^*W^{-1}([\eps,\infty))),
\end{equation}
if necessary redefining $\chi_\zeta$ so that $b$ is smooth. By taking $r$ large, we may ensure that
\begin{equation}\label{e:bbig}b^2 \ge 4 \Gamma q_\zeta^2\end{equation}
everywhere. Note that \eqref{e:bbig} follows from
\[-\{p,q_\zeta^2\} = -2\varphi_\zeta^2\chi_\zeta\chi_\zeta' \ge 4\Gamma\varphi_\zeta^2\chi_\zeta^2 = 4 \Gamma q_\zeta^2\]
on $V_\zeta$.

\textbf{Step 3.} Put $Q = \Op(q_\zeta)$, $B = \Op(b)$, $R = \Op(r)$. Now
\[B^*B = \frac i h [Q^*Q,P] + R + hE,\]
where $P = h^2\Delta_g + V$ and where the error $E= \Op(e)$ has $\supp e \subset \supp q_\zeta$. We have
\[\|Bu\|^2 = \frac i h \la u, [Q^*Q,P] u\ra + \la u, R u \ra + h \la u, E u \ra \le \frac i h \la u, [Q^*Q,P] u\ra + h \|E u\| + \Oh(h^\infty),\]
where we used $\|u\| = 1$ (which was an assumption) and $\|R u\| = \Oh(h^\infty)$ (which follows from Step 1 above, since $R = \Op(r)$ and $r \in C_0^\infty(T^*W^{-1}([\eps,\infty)))$.
Next
\[\begin{split} \frac i h \la u, [Q^*Q,P] u\ra  &= -\frac 2 h \im \la u, Q^*Q(P_1-\lambda) u\ra - \frac 2 h \re\la u, Q^*QW u\ra - \frac 2 h \la u, Q^*Q \im \lambda u\ra\\
&\le - \frac 2 h \re \la u, Q^*[Q,W] u\ra - \frac 2 h \la u, Q^*Q \im \lambda u\ra + \Oh(h^\infty) \end{split},
\]
where we used $\la u, Q^*W Q u\ra \ge 0$ and $\|Q(P_1- \lambda) u\| = \Oh(h^\infty)$ (see \eqref{eq:lemabs0}). We will now show
\begin{equation}\label{e:sub}-\re \la u, Q^*([Q,W] + Q \im \lambda) u\ra \le \frac h 4 \|Bu\|^2 + \Oh(h^2) \|E' u\| + \Oh(h^\infty),\end{equation}
with $E' = \Op(e')$ with $\supp e' \subset \supp q$. Then we will have
\[\|Bu\|^2 \le \Oh(h)(\|E u\| + \|E' u\|) + \Oh(h^\infty),\]
after which an iteration argument, for example as in \cite[Lemma 2]{d}, shows that $\|Bu\| = \Oh(h^\infty)$ allowing us to conclude. The iteration argument involves taking a nested sequence of escape functions $q^j_\zeta$, with corresponding functions $b^j$ as in \eqref{e:bdef}  such that $\supp q^j_\zeta$ is contained in the set where $b^j$ is elliptic (bounded away from $0$). This allows us to show that if $\|Q^j u\| \le C h^k \|u\|$, then $\|B^j u\| \le C h^{k+1/2} \|u\|$.

The estimate \eqref{e:sub} is the slight subtlety discussed in the paragraph preceding the statement of the lemma.
Because $Q^*$ has real principal symbol of order $1$, and $[Q,W]$ has imaginary principal symbol of order $h$, we have
\[|\re \la u, Q^*[Q,W] u\ra| =  \Oh(h^2) \|E' u\| + \Oh(h^\infty),\]
with $E'' = \Op(e'')$ with $\supp e'' \subset \supp q$. Meanwhile
\[\begin{split} \la u,(B^*B + 4h^{-1}\im \lambda Q^*Q)u\ra &\ge \la u, (B^*B - 4\Gamma Q^*Q)u\ra + \Oh(h^\infty) \\
&\ge -C' h\|E'''u\|^2 + \Oh(h^\infty), \end{split}\]
with $E''' = \Op(e''')$ with $\supp e''' \subset \supp q$.
For the second inequality we used the sharp G\aa rding inequality. Indeed, the semiclassical principal symbol of $B^*B - 4\Gamma Q^*Q$ is $b^2 - 4\Gamma q_\zeta^2$, and we may apply \eqref{e:bbig}.
\end{proof}

\subsection{Euclidean ends}\label{s:euc} The model operator near the trapped set is  of the form
\[P_1 = h^2\Delta_g -1\]
off of a compact set $K'$ (which may contain $X_1$) and $(X_1',g)$ is isometric to Euclidean space there. Suppose that each backward bicharacteristic of $h^2\Delta_g$ at energy $\lambda \in [-E,E]$ which enters $T^*(X_1 \cap X_0)$ also enters either $T^*(X_1 \setminus X_0)$ or $T^*(X'_1 \setminus K')$.

In this case, similarly to \S\ref{s:aeinfinity},
the semiclassically outgoing
condition, which is only needed in the Euclidean region
(i.e.\ with backward bicharacteristic disjoint from $K'$), can be proved in several ways. One way is to use an escape function and positive commutator estimate as in Lemma \ref{l:w}: see \cite[Lemma 2]{d} for a complete proof in a more general setting, based on the construction and estimates of \cite{vz}. Another way, which we only outline here, is to show the (off-diagonal)
semiclassical FIO nature of $(P_1-\lambda)^{-1}$ in this region, with Lagrangian given by
the flow-out of the diagonal.
But this follows from the usual parametrix identity, taking some $\chi\in\CI_0(X_1')$
identically $1$ on the compact set, using $G=(1-\chi)\tilde
R_0(\lambda)(1-\chi)$ as the parametrix,
with $\tilde R_0(\lambda)$ the Euclidean resolvent.
Indeed, first for $\im \lambda>0$,
$(P_1-\lambda)G=\Id+E_R$, $G(P_1-\lambda)=\Id+E_L$, with $E_R$ and $E_L$
having Schwartz kernels with support in the left, resp.\ right factor in
$\supp \chi$
(e.g.\ $E_R=-\chi-[P_1,\chi]\tilde R_0(\lambda)$),
so
$$
(P_1-\lambda)^{-1}=G-GE_R+E_L(P_1-\lambda)^{-1}E_R;
$$
this identity
thus also holds for
the analytic continuation. Now, even for the analytic continuation, $G$,
$E_L$ and $E_R$ are semiclassical Lagrangian distributions
away from the diagonal as follows from the explicit formula (where $\im\lambda$
is $\Oh(h)$), and if a point
is in the image of the
wave front relation of $G\chi_0$ or $E_L$ (with $\chi_0$ compactly supported,
identically $1$ on $\supp\chi$) then it is on the forward bicharacteristic
emanating from a point in $T^*\supp \chi_0$,
proving the semiclassically outgoing property of the second and third term
of the parametrix identity.

\subsection{Normally hyperbolic trapped sets}\label{s:wz}

We take these conditions from \cite{wz}. Let $(X_1',g)$ be a manifold which is Euclidean outside of a compact set, let $V \in C_0^\infty(X_1,\R)$, and let
\[P_1 = h^2 \Delta_g + V - 1 - iW,\]
with $W$ as in \eqref{e:abs} and $\supp V \cap \supp W = \emptyset$.

Define the backward/forward trapped sets by
\[\Gamma_{\pm} = \{\zeta \in T^*X_1\colon \mp t \ge 0 \Rightarrow \Phi^t(\zeta) \notin \supp W\},\]
where again $\Phi^t(\zeta) = \exp(tH_p)(\zeta)$.
The trapped set is
\[K \Def \Gamma_+ \cap \Gamma_-.\]
We also define
\[\Gamma_\pm^\lambda = \Gamma_\pm \cap p^{-1}(\lambda), \qquad K_\lambda = K \cap p^{-1}(\lambda).\]

Assume
\begin{enumerate}
\item There exists $\delta > 0$ such that $dp \ne 0$ on $p^{-1}((-\delta,\delta))$.
\item $\Gamma_{\pm}$ are codimension one smooth manifolds intersecting transversely at $K$.
\item The flow is hyperbolic in the normal directions to $K$ in $p^{-1}((-\delta,\delta))$: there exist subbundles $E^\pm_\lambda$ of $T_{K_\lambda}(\Gamma^\lambda_\pm)$ such that 
\[T_{K_\lambda}\Gamma^\lambda_\pm = T K_\lambda \oplus E^\pm_\lambda,\]
where 
\[d \Phi^t \colon E_\lambda^{\pm} \to E_\lambda^{\pm},\]
and there exists $\theta > 0$ such that for all $|\lambda| <\delta,$
\[\|d \Phi^t(v)\| \le C e^{-\theta|t|}\|v\| \textrm{ for all } v \in E^{\mp}_\lambda, \pm t \ge 0.\]
\end{enumerate}
Here and below, by $d \Phi^s$ we mean the differential of
$\Phi^s=\Phi^s(\zeta')$ as a function of $\zeta'$.

This is the normal hyperbolicity assumption which we take from \cite[\S1.2]{wz}. This type of trapping appears in the setting of a slowly rotating Kerr black hole. Under these assumptions we have, from \cite[(1.1)]{wz},
\[\|(P_1 - \lambda)^{-1}\|_{L^2(X'_1) \to L^2(X_1')} \le Ch^{-N},\]
for $\lambda \in [-E,E]-i[0,\Gamma h]$, for some $E,\Gamma, N >0$. In particular, all the assumptions on $P_1$ and $X_1'$ in \S\ref{s:abstract} are satisfied.

\subsection{Trapped sets with negative topological pressure at $1/2$}\label{s:nz}
We take these conditions from \cite{Nonnenmacher-Zworski:Quantum}. Let $(X_1',g)$ be a manifold which is Euclidean outside of a compact set, let $V \in C_0^\infty(X_1,\R)$, and let
\[P_1 = h^2 \Delta_g + V - 1.\]

Let $K_0$ denote the set of maximally extended null-bicharacteristics of $P_1$ which are precompact. We assume that $K_0$ is hyperbolic in the sense that for any $\zeta \in K_0$, the tangent space to $p^{-1}(0)$ (the energy surface) at $\zeta$ splits into flow, unstable, and stable subspaces \cite[Definition 17.4.1]{kh}:
\begin{enumerate}
	\item $T_\zeta(p^{-1}(0)) = \R H_p(\zeta) \oplus E^+_\zeta \oplus E^-_\zeta, \qquad \dim E^\pm_\zeta = \dim X-1$.
	\item $d \Phi^t (E^\pm_\zeta) = E^\pm_{\Phi^t(\zeta)}, \qquad \forall t \in \R$.
	\item $\exists \theta > 0, \|d \Phi^t (v) \| \le C e^{-\theta |t|} \|v\|, \qquad \forall v \in E^\mp_\zeta, \pm t \ge 0.$
\end{enumerate}
Here again $\Phi^t(\zeta)\Def \exp(tH_p)(\zeta)$. This condition is satisfied in the case where $X$ is negatively curved near $K_0$.

The unstable Jacobian $J^u_t(\zeta)$ for the flow at $\zeta$ is given by
\[J^u_t(\zeta) = \det \left(d \Phi^{-t}(\Phi^t(\zeta))|_{E^+_{\Phi^t_\zeta}}\right).\]

We now define the topological pressure $P(s)$ of the flow on the trapped set, following \cite[\S3.3]{Nonnenmacher-Zworski:Quantum} (see also \cite[Definition 20.2.1]{kh}). We say that a set $S \subset K_0$ is $(\eps,t)$ separated if, given any $\zeta_1,\zeta_2 \in S$, there exists $t' \in [0,t]$ such that the distance between $\Phi^{t'}(\zeta_1)$ and $\Phi^{t'}(\zeta_2)$ is at least $\eps$. For any $s \in \R$ define
\[Z_t(\eps,s) \Def \sup_S \sum_{\zeta \in S} (J^u_t(\zeta))^s,\]
where the supremum is taken over all sets $S \subset K_0$ which are $(\eps,t)$ separated. The pressure is then defined as 
\[\mathcal{P}(s) = \lim_{\eps \to 0} \limsup_{t \to \infty} \frac 1 t \log Z_t(\eps,s).\]
The crucial assumption on the dynamics of the bicharacterstic flow on the trapped set is that
\[\mathcal{P}(1/2) < 0.\]
Then from \cite[Theorem 3]{Nonnenmacher-Zworski:Quantum} and \cite[(1.5)]{nz2} we have for any $\Gamma < |P(1/2)|$ and $\chi \in C_0^\infty(X_1')$, there exist $C,E,N>0$ such that
\[\|\chi(P_1 - \lambda)^{-1}\chi\|_{L^2 \to L^2} \le Ch^{-1-N|\im \lambda|/h}\log(1/h),\]
for $\lambda \in [-E,E] - i[0,\Gamma h]$. In particular, all the assumptions on $P_1$ and $X_1'$ in \S\ref{s:abstract} are satisfied.

\subsection{Convex obstacles with negative abscissa of absolute convergence}\label{s:ps}

We take these conditions from \cite{ps}. Let $(X_1',g) = \R^n \setminus O$, where $g$ is the Euclidean metric and where $O = O_1 \cup \cdots \cup O_{k_0}$ is a union of disjoint convex bounded open sets with smooth boundary, and let
\[P_1 = h^2\Delta_g - 1\]
with Dirichlet boundary conditions and. Assume that the $O_j$ satisfy the no-eclipse condition: namely that for each pair $O_i \ne O_j$ the convex hull of $\overline{O_i}$ and $\overline{O_j}$ does not intersect any other $\overline{O_k}$.

In this setting having negative topological pressure at $1/2$ is equivalent to having negative abscissa of convergence of a certain dynamical zeta function, a condition under which a holomorphic continuation to strip of a polynomially bounded cutoff resolvent was first obtained by Ikawa \cite{i}. To define this, for $\gamma$ a primitive periodic reflecting ray with $m_\gamma$ reflections, let $T_\gamma$ be the length of $\gamma$ and $P_\gamma$ the associated linear Poincar\'e map. Let $\lambda_{i,\gamma}$ for $i = 1, \dots, n-1$ be the eigenvalues of $P_\gamma$ with $|\lambda_{i,\gamma}| > 1$. Let $\mathcal{P}$ be the set of primitive periodic rays. Set
\[\delta_\gamma = - \frac 12 \log(\lambda_{1,\gamma}\cdots \lambda_{n-1,\gamma}), \quad\gamma \in \mathcal{P}.\]
Let $r_\gamma = 0$ if $m_\gamma$ is even and $r_\gamma = 1$ if $m_\gamma$ is odd. The dynamical zeta function is given by
\[Z(s) = \sum_{m=1}^\infty \frac 1 m \sum_{\gamma \in \mathcal{P}} (-1)^{mr_\gamma} e^{m(-sT_\gamma + \delta_\gamma)},\]
and the abscissa of convergence is the minimal $s_0 \in \R$ such that the series is absolutely convergent for $\re s > s_0$. Assume that
\[s_0 < 0.\]

For simplicity, assume in addition that $n=2$. This assumption can be replaced by another which is weaker and more dynamical but also more complicated: see \cite[Theorem 1.3]{ps} for a better statement.  Then from \cite[Theorem 1.3]{ps} we have for any $\chi \in C_0^\infty(X_1')$
\[\|\chi(P_1 - \lambda)^{-1}\chi\|_{L^2 \to L^2} \le Ch^{-N}\]
for $\lambda \in [-E,E]-i[0,\Gamma h]$, for some $N,E,m,C >0$ and $\Gamma > |s_0|$. In particular, all the assumptions on $P_1$ and $X_1'$ in \S\ref{s:abstract} are satisfied.

\section{Applications}\label{s:apps}

We now give an improved version of Theorem \ref{t:intro}.

\begin{thm} Let $(X,g)$ be even and asymptotically hyperbolic, let $V \in C_0^\infty(X;\R)$, and let
\[
P = h^2 \Delta_g + V - 1, \qquad p = |\xi|^2_g + V - 1.
\]
\begin{enumerate}
	\item Suppose for some $E>0$ $P$ has a normally hyperbolic trapped set on $p^{-1}[-E,E]$ in the sense of \S\ref{s:wz}. Then there exist $h_0,N,\Gamma, C >0$ such that
	\[\|x^{5/2+\Gamma/2} R(\lambda)x^{5/2 + \Gamma/2}\|_{L^2 \to L^2} \le C h^{-N}\]
	for $\lambda \in [-E,E]-i[0,\Gamma h]$ and $0 < h \le h_0$.
	\item Suppose $P$ has a hyperbolic trapped set on $p^{-1}(0)$with $\mathcal{P}(1/2)<0$ as in \S\ref{s:nz}. Then for any $\Gamma < |\mathcal{P}(1/2)|$ there exist $E,h_0,N,C>0$ such that
	\[\|x^{5/2+\Gamma/2} R(\lambda)x^{5/2 + \Gamma/2}\|_{L^2 \to L^2} \le C h^{-N}\]
	for $\lambda \in [-E,E]-i[0,\Gamma h]$ and $0 < h \le h_0$.
	\item Let
	\[(\tilde X,g) = (\R^2, dr^2 + f(r)d\theta^2)\]
	with $f \in C^\infty((0,\infty);(0,\infty))$ has $f(r) = r^2$ for $r$ sufficiently small, $f(r) = \sinh^2(r)$ for $r$ sufficiently large, and $f'(r) > 0$ for all $r$. Let $X = \tilde X \backslash O$ where $O$ is a union of disjoint convex open sets all contained in the region where $f(r) = r^2$ , satisfying the no-eclipse condition, with abscissa of convergence $s_0 < 0$ as in \S\ref{s:ps}, and with Dirichlet boundary conditions imposed for $P = h^2\Delta_g - 1$. Then there exist $E,h_0,N > 0$ and $\Gamma > |s_0|$ such that
	\[\|x^{5/2+\Gamma/2} R(\lambda)x^{5/2 + \Gamma/2}\|_{L^2 \to L^2}\le C h^{-N}\]
	for $\lambda \in [-E,E]-i[0,\Gamma h]$ and $0 < h \le h_0$.
\end{enumerate}
\end{thm}
Note that in case (3) we assume $f'>0$ to rule out geometric trapping and to guarantee \eqref{eq:convexity}.  Here $\exp(-(1+r^2)^{1/2})$, perhaps multiplied by a suitable large constant prefactor, playes the role of the boundary defining function $x$. The set $X_1$ encompasses the whole region where $f(r) \not\equiv \sinh^2(r)$, and the set $X_0$ is contained in the region where $f(r) = \sinh^2(r)$.

The theorem follows immediately from the main theorem, Theorem~\ref{t:main},
together with \S\ref{s:ahinfinity} (in which we show that the assumptions on $P_0$
are satisfied and derive the weights $x^{5/2 + \Gamma/2}$), and \S\ref{s:wz}, resp.\ \S\ref{s:nz}, resp.\ \S\ref{s:ps} in the cases
(1), resp.\ (2), resp.\ (3) (in which we show that the assumptions on $P_1$ are satisfied).

\begin{rem}
The same results also hold when $(X,g)$ has Euclidean ends in the sense of \S\ref{s:aeinfinity}; one merely needs to use the results of \S\ref{s:aeinfinity} instead of
those of \S\ref{s:ahinfinity}. In this case the (mild) difference with previous authors is that we obtain the analytic continuation and the resolvent estimates for the resolvent with exponential weights (as in \S\ref{s:aeinfinity}) rather than with compactly supported cutoff functions. Note however that in \cite{bp} Bruneau-Petkov give a method for passing from cutoff resolvent estimates to weighted resolvent estimates in such a situation.
\end{rem}

In the special case where $V \equiv 0$ and
\[H = \Delta_g  - \frac{(n-1)^2}4,\]
we can obtain a resonant wave expansion as a corollary. Indeed, we
we have the resolvent estimate
\begin{equation}\label{e:resestwave}
\|x^{5/2+\Gamma/2} R(z)x^{5/2 + \Gamma/2}\|_{L^2 \to L^2} \le C|z|^{N-2}, \qquad |\re z|>z_0, \quad \im z > -|\Gamma|/2,
\end{equation}
where $R(z)$ is now $(H - z^2)^{-1}$ for $\im z > 0$ or its meromorphic continuation for $\im z <0$. This follows from the substitution
\[h^2z^2 = 1 + \lambda, \qquad \re z = h^{-1}.\]
On the other hand, work of Mazzeo-Melrose \cite{mm} and Guillarmou \cite{g} (see also \cite{v1,v2}) shows that the  weighted resolvent $x^{5/2+\Gamma/2} R(z)x^{5/2 + \Gamma/2}$ continues meromorphically to $\{\im z > -|\Gamma|/2\}$. From this the following resonant wave expansion follows.

\begin{cor} Suppose $u$ solves
\begin{equation}\label{e:wave}
(\D_t^2 + H)u = 0, \qquad u|_{t=0}=f, \quad\D_t u|_{t=0}=g
\end{equation}
for $f,g \in C_0^\infty(X)$, with support disjoint from the convex obstacles in the case (3) above (and with no restriction on the support in cases (1) and (2)). Then
\begin{equation}\label{e:resexpand}
u(t) =  \sum_{\im z_j > -\Gamma/2}\sum_{m=0}^{M(z_j)} e^{-itz_j}t^m w_{z,j,m} + E(t,x).
\end{equation}
The sum is taken over poles of $R(z)$, $M(z_j)$ is the algebraic multiplicity of the pole at $z_j$, and the $w_{z,j,m} \in C^\infty(X)$ are eigenstates or resonant states. The error term obeys the estimate
\[|\D^\alpha E(t,x)| \le C_{\alpha,\epsilon} e^{-t(\Gamma/2 - \epsilon)}\]
for every $\epsilon > 0$ and multiindex $\alpha$, uniformly over compact subsets of $X$.
\end{cor}

Note that the sum in \eqref{e:resexpand} is finite, thanks to the resonance free strip established by the estimate \eqref{e:resestwave}. This is a standard consequence of the resolvent estimate and the meromorphic continuation, by taking a Fourier transform in time and then performing a contour deformation. See for example \cite[\S 6.3]{d2} for a similar result, and \cite[\S 4]{msv2} for a similar result with an asymptotic extending to infinity in space. We sketch the proof here: see \cite[\S 6.3]{d2} for more details.  When $f \equiv 0$, we can write
\[
u(t) = \frac 1 {2\pi} \int_{-\infty + i K}^{\infty + i K} e^{-iz t}R(z) g \,dz,
\]
where $K > (n-1)/2$. The proof then proceeds by contour deformation from $\{\im z = K\}$ to $\{\im z = -\Gamma/2\}$. The residues at the poles of the resolvent produce the terms of the expansion in \eqref{e:resexpand}, and the resolvent estimate
\begin{equation}\label{e:resestwave2}
\|x^{5/2+\Gamma/2} R(z)x^{5/2 + \Gamma/2}\|_{H^{s+N} \to H^s} \le C|z|^{-2},
\end{equation}
for any $s$, justifies the deformation and controls the $H^s$ norm of the error (on compact sets, or in suitably weighted spaces) in terms of the $H^{s+N}$ norm of $g$. The estimate \eqref{e:resestwave2} can be derived from the $L^2 \to L^2$ estimate \eqref{e:resestwave} following the same procedure as in \S\ref{s:ahinfinity} above. The case where $f \not \equiv 0$ and $g \equiv 0$ can be deduced similarly by differentiating the equation \eqref{e:wave} in $t$, and then the general case follows by superposition of these two cases.

In many settings better resolvent estimates are available in the physical half plane $\im \lambda > 0$. More specifically, we obtain the following theorem (see \cite[(1.1)]{wz} and \cite[(1.17)]{Nonnenmacher-Zworski:Quantum} for the corresponding resolvent estimates for the trapping model operators).

\begin{thm}\label{t:logloss} Let $(X,g)$ be even and asymptotically hyperbolic, let $V \in C_0^\infty(X;\R)$, and let
\[
P = h^2 \Delta_g + V - 1.
\]
 Suppose $P$ has a normally hyperbolic trapped set in the sense of \S\ref{s:wz} or a hyperbolic trapped set with $\mathcal{P}(1/2)<0$ as in \S\ref{s:nz}. Then for any $\chi \in C_0^\infty(X)$ there exist $E,h_0,C >0$ such that
	\[\|\chi R(\lambda)\chi\|_{L^2 \to L^2} \le C \log(1/h)h^{-1}\]
	for $\lambda \in [-E,E]+i[0,\infty)$, $0 < h \le h_0$.
\end{thm}

In \cite{Bony-Burq-Ramond}, Bony-Burq-Ramond prove that for $P$ a semiclassical Schr\"odinger operator on $\R^n$, the presence of a single trapped trajectory implies that
\[ \log(1/h)h^{-1}\le C\sup_{\lambda \in [-\eps,\eps]} \|\chi R(\lambda)\chi\|,\]
provided $\chi \in C_0^\infty(X)$ is $1$ on the projection of the trapped set. Consequently, in that setting (and probably in general), Theorem~\ref{t:logloss} is optimal.

From Theorem~\ref{t:logloss} it follows by a standard $TT^*$ argument as in \cite[\S6]{d} that the Schr\"odinger propagator exhibits local smoothing with loss:
\[\int_0^T \|\chi e^{-it\Delta_g}u\|^2_{H^{1/2 - \eps}}dt \le C_{T,\eps}\|u\|^2_{L^2},\]
for any $T,\eps > 0$. In fact, the main resolvent estimate of \cite{d} follows from Theorem \ref{t:main} above, because the model operator near infinity, $P_0$ can be taken to be a nontrapping scattering Schr\"odinger operator, for which the necessary resolvent and propagation estimates were proved in \cite{vz}. Moreover, Burq-Guillarmou-Hassell \cite{bgh} show that when $\mathcal{P}(1/2) < 0$ semiclassical resolvent estimates with logarithmic loss can be used to deduce Strichartz estimates with no loss on a scattering manifold (a manifold with asymptotically Euclidean or asymptotically conic ends in a sense which generalizes that of \S\ref{s:aeinfinity}), and the same result probably holds on the asymptotically hyperbolic spaces considered here. See also \cite{bgh} for more references and a discussion of the history and of recent developments in local smoothing and Strichartz estimates.

Another possible application of the method is to give alternate proofs of cutoff resolvent estimates in the presence of trapping, where the support of the cutoff is disjoint from the trapping. As mentioned in the introduction, estimates of this type were proved by Burq \cite{Burq:Lower} and Cardoso and Vodev \cite{Cardoso-Vodev:Uniform} and take the form
\[\|\chi R(\lambda)\chi\|_{L^2 \to L^2} \le Ch^{-1},\]
for $\im \lambda > 0$, where $\chi \in C^\infty(X)$ vanishes on the convex hull of the trapped set and is either compactly supported or suitably decaying near infinity. Indeed a related method based on propagation of singularities was used in \cite{dv2} to prove such a result.

\end{document}